\documentclass[12pt, oneside]{amsart}
\usepackage{graphicx}
\usepackage{amsfonts}
\usepackage{epsf}
\usepackage{amssymb}
\usepackage{amsmath}
\usepackage{amscd}
\usepackage{amsthm}
\usepackage{tikz}
\usepackage{pdfpages}
\usepackage{fancyhdr}
\usepackage{setspace}
\usepackage{hyperref}
\usepackage[all]{xy}
\usetikzlibrary{matrix}
\usepackage{verbatim}
\usepackage{enumerate}
\usepackage{mathrsfs}
\usepackage{pinlabel}
\usepackage{lscape}
\usepackage{color}
\usepackage[colorinlistoftodos]{todonotes}

\newtheorem{theorem}{Theorem}
\newtheorem{proposition}[theorem]{Proposition}
\newtheorem{lemma}[theorem]{Lemma}

\newtheorem*{question}{Question}

\newcommand{\Z}{\mathbb{Z}}
\newcommand{\C}{\mathbb{C}}

\newcommand{\Q}{\mathbb{Q}}

\def\co{\colon\thinspace}

\makeatletter
\newtheorem*{rep@theorem}{\rep@title}
\newcommand{\newreptheorem}[2]{%
\newenvironment{rep#1}[1]{%
 \def\rep@title{#2 \ref{##1}}%
 \begin{rep@theorem}}%
 {\end{rep@theorem}}}
\makeatother

\newreptheorem{theorem}{Theorem}
\newreptheorem{lemma}{Lemma}
\newreptheorem{question}{Question}
\newreptheorem{corollary}{Corollary}

\begin{document}
\makeatletter
\providecommand\@dotsep{5}
\makeatother
\rhead{\thepage}
\lhead{\author}
\thispagestyle{empty}


\raggedbottom
\pagenumbering{arabic}
\setcounter{section}{0}


\title{Brieskorn spheres bounding rational balls}
\author{Selman Akbulut \and Kyle Larson}
\address{Michigan State University, East Lansing, Michigan}
\email{akbulut@math.msu.edu}
\address{Michigan State University, East Lansing, Michigan}
\email{larson@math.msu.edu}

\begin{abstract}
Fintushel and Stern showed that the Brieskorn sphere $\Sigma(2,3,7)$ bounds a rational homology ball, while its non-trivial Rokhlin invariant obstructs it from bounding an integral homology ball. It is known that their argument can be modified to show that the figure-eight knot is rationally slice, and we use this fact to provide the first additional examples of Brieskorn spheres that bound rational homology balls but not integral homology balls: the families $\Sigma(2,4n+1,12n+5)$ and $\Sigma(3,3n+1,12n+5)$ for $n$ odd. We also provide handlebody diagrams for a rational homology ball containing a rationally slice disk for the figure-eight knot, as well as for a rational homology ball bounded by $\Sigma(2,3,7)$. These handle diagrams necessarily contain 3-handles.
\end{abstract}
\maketitle


\begin{section}{Introduction}\label{introduction}


A classic question in low-dimensional topology asks which 3-dimensional integral homology spheres smoothly bound integral homology balls. Due to their nice properties, a reasonable starting point to address this question is to consider the Brieskorn homology spheres $\Sigma(p,q,r) = \{x^p + y^q+z^r = 0\} \cap S^5 \subset \C^3$, with $p,q,r$ positive and relatively prime.  A large number of Brieskorn spheres are known to bound integral homology balls (or even contractible 4-manifolds), for example see \cite{AK}, \cite{CH}, \cite{Stern}, \cite{FS3}, and \cite{Fickle}. One can weaken the above question to ask which integral homology spheres bound \emph{rational} homology balls; however, it turns out that this has not helped much in producing more examples. Indeed it appears to be a difficult problem to find Brieskorn spheres that bound rational homology balls but not integral homology balls. Fintushel and Stern \cite{FS} provided the first example by constructing a rational homology ball bounded by the Brieskorn sphere $\Sigma(2,3,7)$. Since $\Sigma(2,3,7)$ has Rokhlin invariant $\mu = 1$, it cannot bound an integral homology ball. In this note we give the first new examples.
Using the fact that the figure-eight knot is rationally slice (see Section~\ref{handles}), a simple observation shows that the $\mu = 1$ Brieskorn sphere $\Sigma(2,3,19)$ bounds a rational homology ball (Proposition \ref{2319}). More interesting are two infinite families of Brieskorn spheres.

\begin{theorem}\label{main}
The Brieskorn spheres $\Sigma(2,4n+1,12n+5)$ and $\Sigma(3,3n+1,12n+5)$ bound rational homology balls, and when $n$ is odd they have $\mu = 1$ and so do not bound integral homology balls.
\end{theorem}

The difficulty of this problem is related to handle decompositions of 4-manifolds. A simple homological argument shows that if an integral homology 3-sphere bounds a rational homology ball $X$ that is not an integral homology ball, then any handle decomposition of $X$ must contain 3-handles. Therefore the difficulty of working with handle decompositions with 3-handles points to the challenge and interest of finding these examples.

\vspace{.05in}

One of the main motivations for our work relates to the group of integral homology spheres $\Theta^3_\Z$ and the group of rational homology spheres $\Theta^3_\Q$. There is a canonical homomorphism $\psi \co \Theta^3_\Z \rightarrow \Theta^3_\Q$ induced by inclusion (see \cite{AL} for more discussion of this homomorphism), and the work of \cite{FS} can be interpreted as showing that the kernel of $\psi$ is non-trivial. Since $\Sigma(2,3,7)$ has infinite order in $\Theta^3_\Z$ we get that the kernel contains a subgroup isomorphic to $\Z$. Beyond this nothing is known of its structure, but it seems likely that the kernel is in fact much larger. Theorem \ref{main} gives a large collection of additional Brieskorn spheres that represent non-trivial elements in the kernel of $\psi$, but it is unknown if they are linearly independent in $\Theta^3_\Z$. Since Brieskorn spheres are often amenable to the computation of gauge and Floer theoretic invariants, it is possible that the following is a tractable question.

\begin{question}
Is some subset of the Brieskorn spheres $\Sigma(2,3,7)$, $\Sigma(2,3,19)$, $\Sigma(2,4n+1,12n+5)$, and $\Sigma(3,3n+1,12n+5)$ linearly independent in $\Theta^3_\Z$?
\end{question}

Now we outline the proof of Theorem \ref{main}. Recall that $\Sigma(2,3,7)$ can be obtained by $+1$-surgery on the figure-eight knot, and in fact the construction in \cite{FS} can be modified to show that the figure-eight knot is rationally slice, that is, bounds a smooth disk in a rational homology ball bounded by $S^3$ (see \cite{Cha}). This fact was apparently also known by Kawauchi\footnote{Kawauchi adds an extra algebraic condition in his definition of rationally slice and shows that the (2,1)-cable of the figure-eight knot satisfies this stronger condition. However, it is implicit in his argument that the figure-eight knot satisfies the weaker definition we use.} \cite{Kawauchi1}, whose argument moreover generalizes to show that all strongly negative-amphicheiral knots are rationally slice (see \cite{Kawauchi2}, \cite{KimWu}). In Section~\ref{handles} we give a handle proof that the figure-eight knot is rationally slice. Our proof uses the fact that the linking circle (meridian) of the 1-handle of $-W^{+}(0,2)$  is slice in $-W^{+}(0,2)$
(see p. 23 of \cite{Akbulut} for the notation).   
 
\vspace{.05in}

 If $Y$ denotes the 3-manifold obtained by 0-surgery on the figure-eight knot, it then follows that $Y$ bounds a 4-manifold with the rational homology of $S^1 \times D^3$. Hence any homology sphere obtained by integral surgery on $Y$ will bound a rational homology ball (the surgery corresponds to attaching a 2-handle to the rational homology $S^1 \times D^3$ which kills the non-torsion part of the homology). Theorem \ref{main} is then proved by showing that each $\Sigma(2,4n+1,12n+5)$ and $\Sigma(3,3n+1,12n+5)$ can be obtained by an integral surgery on $Y$. We do this in Section \ref{proof}.

\vspace{.05in}

In Section \ref{handles} we give handle diagrams for some of the relevant rational homology balls. In particular we give handle diagrams for a rational homology ball bounded by $\Sigma(2,3,7)$, and a rational homology ball bounded by $S^3$ showing a rationally slice disk for the figure-eight knot. While the arguments in \cite{FS} and \cite{Cha} are constructive, they do not give explicit handle diagrams for the rational homology balls they construct.

\subsection*{Acknowledgements} The second author thanks Paolo Aceto and Ron Stern for helpful conversations at the beginning of this project, and especially Ron Stern for suggesting Lemma \ref{construction} as a way to construct new examples.
\end{section}

\begin{section}{Proof of Theorem 1}\label{proof}

We start with the following lemma whose proof was outlined in Section \ref{introduction}. As before, let $Y$ denote 0-surgery on the figure-eight knot.

\begin{lemma}\label{construction}
Any integral homology sphere obtained by integral surgery on $Y$ bounds a rational homology ball.
\end{lemma}

\begin{proof}
Let $X$ be a rational homology ball with boundary $S^3$ such that the figure-eight knot bounds a smooth properly embedded disk $D$ in $X$ (for example, see Section \ref{handles}). Then the Mayer-Vietoris long exact sequence shows that $C:= X \setminus \nu D$ has the rational homology of $S^1 \times D^3$. An integral surgery on $Y$ corresponds to attaching to a 2-handle to $C$. If the resulting 3-manifold is an integral homology sphere, then the 4-manifold $W$ obtained by attaching the corresponding 2-handle to $C$ must be a rational homology ball, as can be seen from the Mayer-Vietoris long exact sequence and the long exact sequence of the pair $(W,\partial W)$.
\end{proof}

We remark that in the previous lemma we can use 0-surgery on any rationally slice knot, and not just the figure-eight knot.

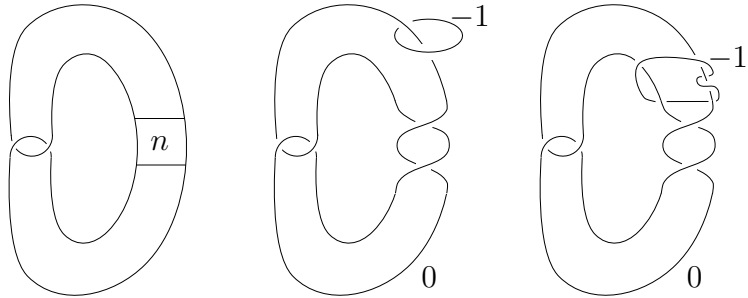
\begin{figure}
\labellist
\pinlabel $n$ at 410 410
\pinlabel $0$ at 1140 50
\pinlabel $-1$ at 1250 750
\pinlabel $0$ at 1860 50
\pinlabel $-1$ at 1950 640
\endlabellist
\centering
\includegraphics[scale=0.15]{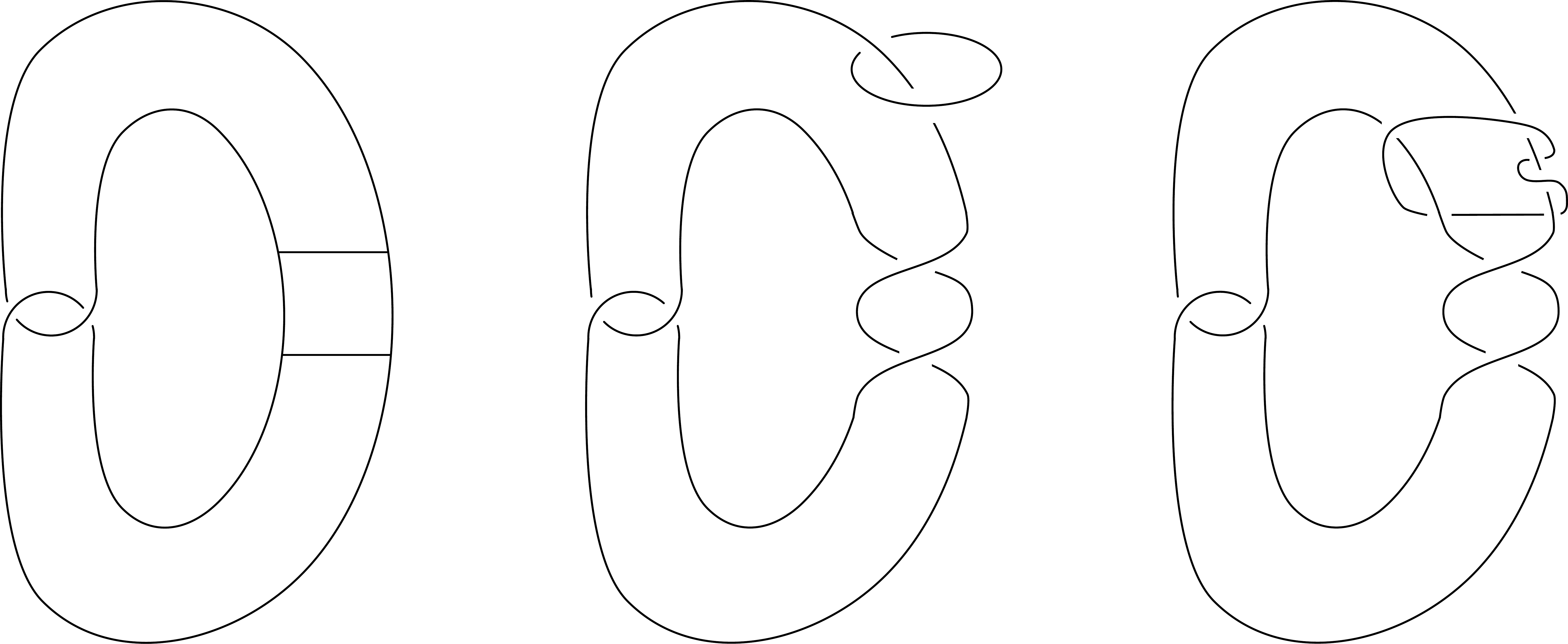}
\caption{On the left we have the knot $K_n$, where the box denotes $n$ full right-handed twists. In the middle is $\Sigma(2,3,7)$, and on the right is $\Sigma(2,3,19)$.}\label{twist}
\end{figure}

\begin{proposition}\label{2319}
$\Sigma(2,3,19)$ bounds a rational homology ball.
\end{proposition}

\begin{proof}
Using Lemma \ref{construction} it suffices to show that $\Sigma(2,3,19)$ can be obtained by an integral surgery on $Y$. The Brieskorn sphere $\Sigma(2,3,6n+1)$ admits a surgery description as $+1$-surgery on the twist knot $K_n$ defined as in Figure \ref{twist} (for this and plumbing descriptions of Brieskorn spheres see \cite{Saveliev}). Note that $K_1$ is the figure-eight knot. Blowing down the $-1$-framed components in the middle and right pictures of Figure \ref{twist} results in $+1$-surgery on $K_1$ and $K_3$, respectively, showing that $\Sigma(2,3,7)$ and $\Sigma(2,3,19)$ bound rational homology balls.
\end{proof}

\begin{proof}[Proof of Theorem \ref{main}]
For the families $\Sigma(2,4n+1,12n+5)$ and $\Sigma(3,3n+1,12n+5)$ we use the dual approach, giving integral surgeries from their canonical negative definite plumbings to $Y$. For $\Sigma(2,4n+1,12n+5)$ these plumbings take the form
\[
\xygraph{
!{<0cm,0cm>;<1cm,0cm>:<0cm,1cm>::}
!~-{@{-}@[|(2.5)]}
!{(0,1.5) }*+{\bullet}="x"
!{(1.5,3) }*+{\bullet}="a1"
!{(4.5,3) }*+{\bullet}="a2"
!{(1.5,0) }*+{\bullet}="c1"
!{(1.5,1.5) }*+{\bullet}="b1"
!{(3,1.5) }*+{\bullet}="b2"
!{(6,1.5) }*+{\bullet}="b3"
!{(3,3) }*+{\bullet}="am"
!{(4.5,1.5) }*+{\dots}="bm"
!{(4.5,2.5)}*+{\overbrace{\hphantom{--------}}^{n-1}}
!{(1.5,1.9) }*+{-5}
!{(3,1.9) }*+{-2}
!{(6,1.9) }*+{-2}
!{(0,1.9) }*+{-1}
!{(1.5,3.4) }*+{-4}
!{(4.5,3.4) }*+{-3}
!{(1.5,0.4) }*+{-2}
!{(3.1,3.4) }*+{-n-1}
"x"-"c1"
"x"-"a1"
"x"-"b1"
"b1"-"b2"
"a1"-"am"
"b2"-"bm"
"a2"-"am"
"b3"-"bm"
},
\]
and for $\Sigma(3,3n+1,12n+5)$ we have
\[
\xygraph{
!{<0cm,0cm>;<1cm,0cm>:<0cm,1cm>::}
!~-{@{-}@[|(2.5)]}
!{(0,1.5) }*+{\bullet}="x"
!{(1.5,3) }*+{\bullet}="a1"
!{(4.5,3) }*+{\bullet}="a2"
!{(1.5,0) }*+{\bullet}="c1"
!{(1.5,1.5) }*+{\bullet}="b1"
!{(3,1.5) }*+{\bullet}="b2"
!{(6,1.5) }*+{\bullet}="b3"
!{(3,3) }*+{\bullet}="am"
!{(4.5,1.5) }*+{\dots}="bm"
!{(4.5,2.5)}*+{\overbrace{\hphantom{--------}}^{n-1}}
!{(1.5,1.9) }*+{-4}
!{(3,1.9) }*+{-2}
!{(6,1.9) }*+{-2}
!{(0,1.9) }*+{-1}
!{(1.5,3.4) }*+{-3}
!{(4.5,3.4) }*+{-4}
!{(1.5,0.4) }*+{-3}
!{(3.1,3.4) }*+{-n-1}
"x"-"c1"
"x"-"a1"
"x"-"b1"
"b1"-"b2"
"a1"-"am"
"b2"-"bm"
"a2"-"am"
"b3"-"bm"
}.
\]

When $n=1$ we get $\Sigma(2,5,17)$ and $\Sigma(3,4,17)$. Surgery diagrams for their canonical negative definite plumbings appear as the gray components in Figure \ref{2517} and Figure \ref{3417}, and the black component gives the necessary surgery to $Y$. This can be seen from a straightforward sequence of blowdowns which we leave to the reader.

Now we describe an iterative procedure to obtain the whole families. Starting with the plumbing for either $\Sigma(2,5,17)$ or $\Sigma(3,4,17)$, in Figure \ref{2517} or Figure \ref{3417} we can blow up to unlink the black $-1$-framed component from the $-2$-framed component. This is demonstrated in Figure \ref{iterate}. The result is to lower the framing of the $-2$-framed component to $-3$, and the previous surgery curve becomes a $-2$-framed component in the bottom chain. If we forget about the newly introduced $-1$-framed unknot we see a plumbing for the $n=2$ case, and the $-1$-framed unknot again provides the required surgery to $Y$ since blowing up does not change the boundary 3-manifold. It is clear that we can keeping blowing up in this fashion, each time adding a $-2$-framed component to the bottom chain and decreasing the framing on the appropriate component in the upper-right chain by 1. Hence if we start with $\Sigma(2,5,17)$, blowing up in this way will generate the family $\Sigma(2,4n+1,12n+5)$, and starting with $\Sigma(3,4,17)$ generates $\Sigma(3,3n+1,12n+5)$. In each case the $-1$-framed unknot coming from the blow up provides the surgery to $Y$, and so by Lemma \ref{construction} these Brieskorn spheres bound rational homology balls.

It is easy to compute the Neumann-Siebenmann $\bar{\mu}$ invariant for our examples from the diagrams of their canonical negative definite plumbings, as described for example in \cite{Neumann}. The plumbings have signature $-5-n$, and when $n$ is odd the square of their spherical Wu class is $-13-n$. Hence for odd $n$ these Brieskorn spheres have $\bar{\mu}=(-5-n)-(-13-n)=8$. Since $\mu = \frac{\bar{\mu}}{8} \mod(2)$, we see that these examples have nontrivial Rokhlin invariant.


\end{proof}

\begin{figure}
\labellist
\pinlabel $-2$ at 90 280
\pinlabel $-1$ at 210 280
\pinlabel $-4$ at 330 280
\pinlabel $-2$ at 450 280
\pinlabel $-3$ at 570 280
\pinlabel $-5$ at 90 70
\pinlabel $-1$ at 330 65
\endlabellist
\centering
\includegraphics[scale=0.3]{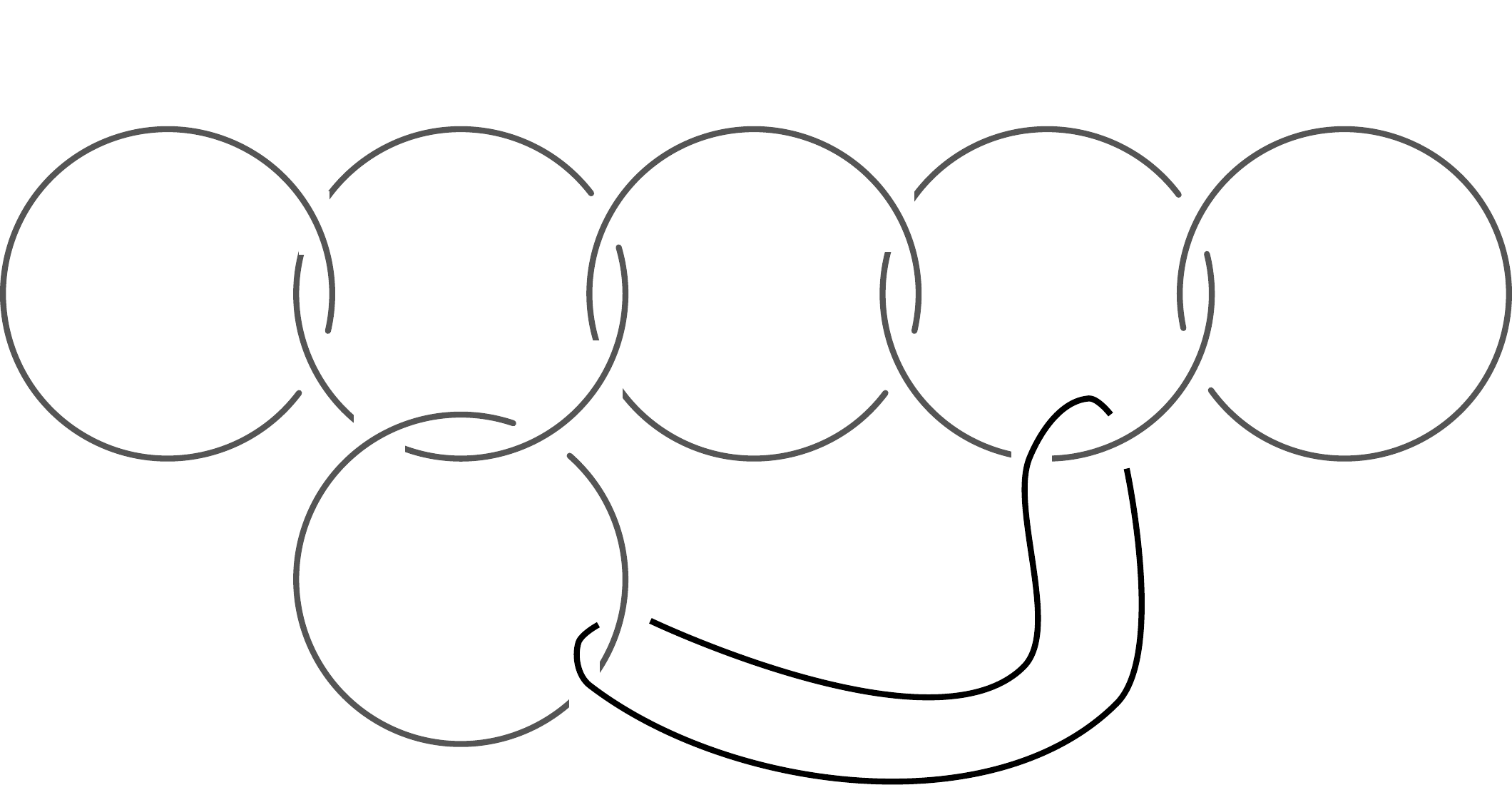}
\caption{An integral surgery from $\Sigma(2,5,17)$ to $Y$.}\label{2517}
\end{figure}
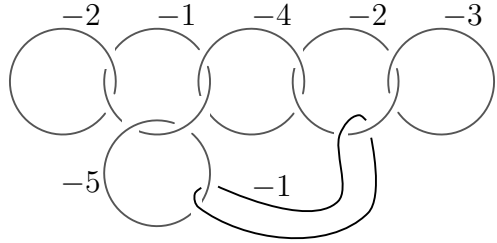

\begin{figure}
\labellist
\pinlabel $-3$ at 90 280
\pinlabel $-1$ at 210 280
\pinlabel $-3$ at 330 280
\pinlabel $-2$ at 450 280
\pinlabel $-4$ at 570 280
\pinlabel $-4$ at 90 70
\pinlabel $-1$ at 330 65
\endlabellist
\centering
\includegraphics[scale=0.3]{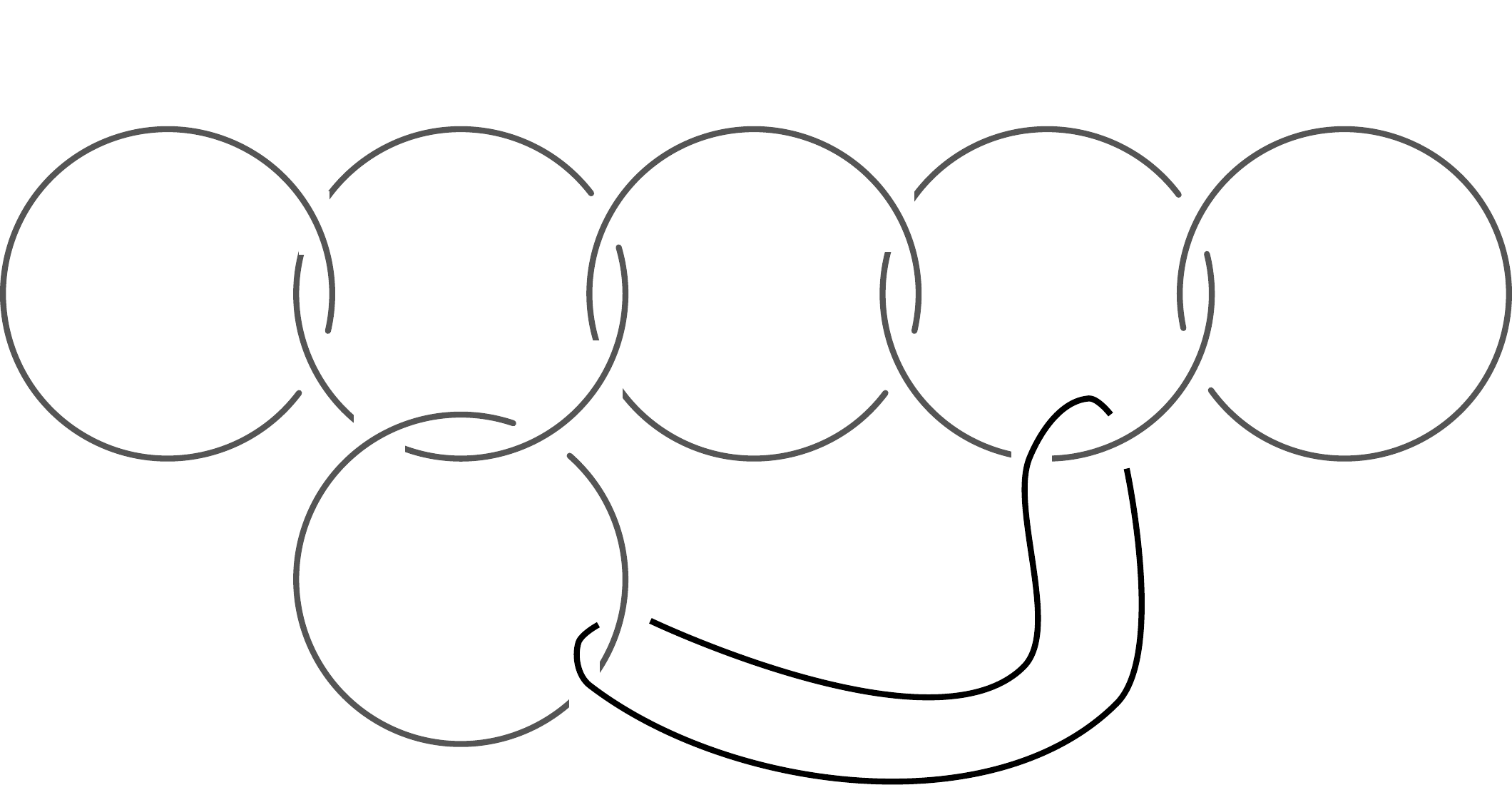}
\caption{An integral surgery from $\Sigma(3,4,17)$ to $Y$.}\label{3417}
\end{figure}

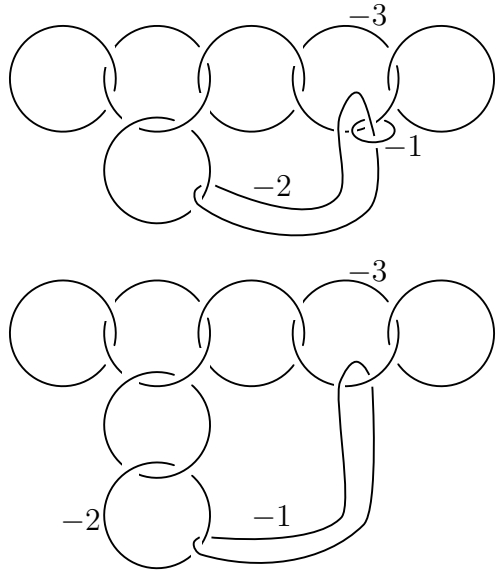
\begin{figure}
\labellist
\pinlabel $-2$ at 330 480
\pinlabel $-3$ at 450 695
\pinlabel $-3$ at 450 374
\pinlabel $-1$ at 495 530
\pinlabel $-2$ at 90 60
\pinlabel $-1$ at 330 65
\endlabellist
\centering
\includegraphics[scale=0.3]{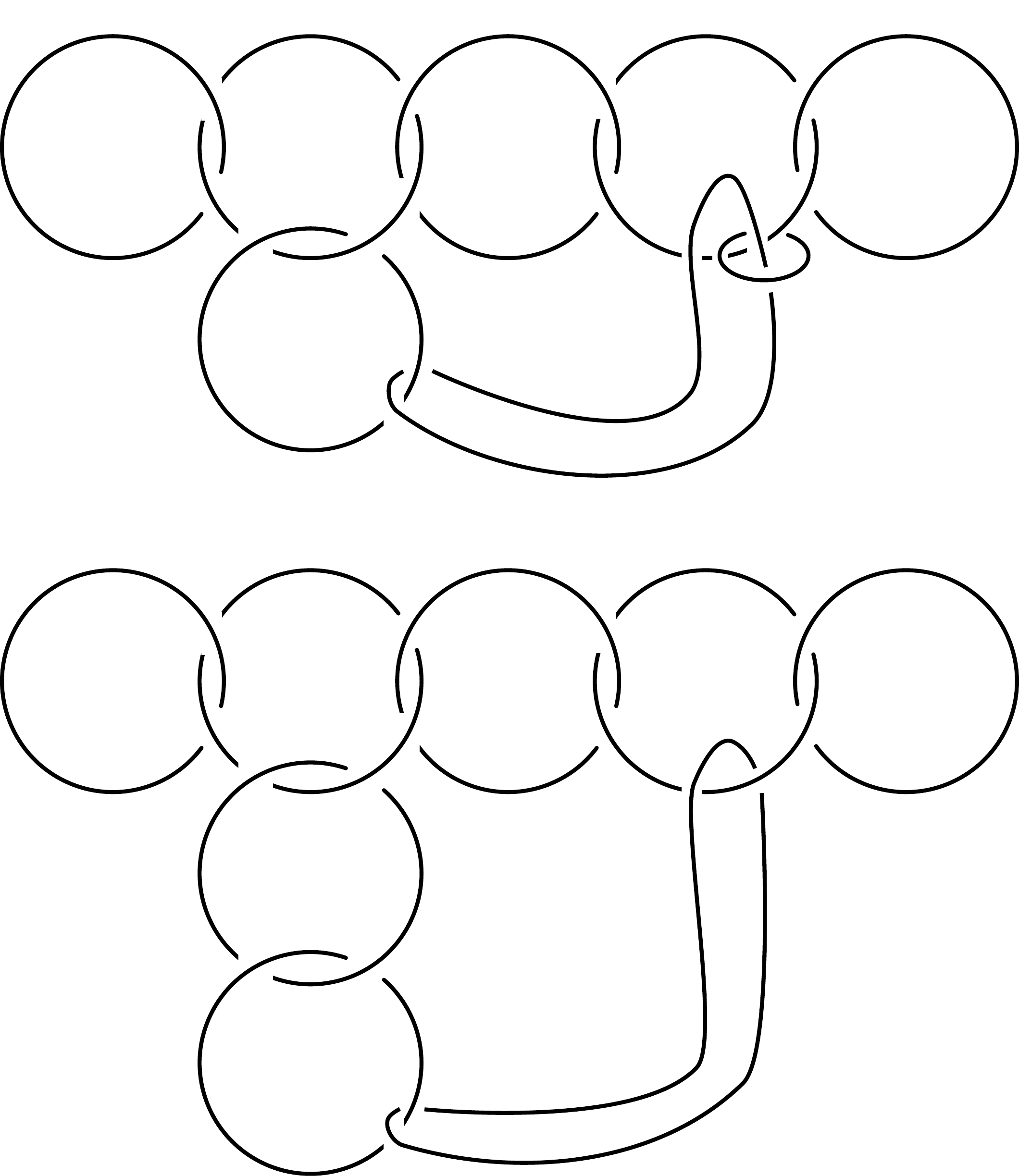}
\caption{Blowing up in the diagram.}\label{iterate}
\end{figure}

\end{section}

\begin{section}{Handle diagrams}\label{handles}

Here we will first construct a rational ball $W_{0}$ with boundary $S^3$ where the figure-eight knot $K$ is slice.  Then by blowing down this slice disk we will get a specific handlebody of a rational ball $W$ which $\Sigma(2,3,7)$ bounds. We start with the 
figure-eight knot in $S^{3}=\partial B^{4}$, drawn as the yellow curve in Figure~\ref{r1}. Then we attach a canceling $2/3$ handle pair to $B^4$. After this, we apply the obvious boundary diffeomorphisms to get the last picture of Figure~\ref{r1}. Then we go from Figure~\ref{r1} to Figure~\ref{r2} by applying the local diffeomorphism described in Figure~\ref{r1a}. This brings us to the first picture of Figure~\ref{r2}. Then an isotopy and handle slide (indicated by the dotted arrow) and turning a zero framed $2$-handle to a dotted circle gives us the last picture of Figure~\ref{r2}, which is a rational ball and the figure-eight knot $K$ (drawn in yellow color) is obviously slice there. Then by blowing down this slice $K$,  in Figure~\ref{r3} with $+1$-framing, gives a rational ball $W$  which $\Sigma(2,3,7)$ bounds. The notation in Figure~\ref{r3} means that everything going through $K$ is twisted by a $-1$-twist. From the picture we see that $W$ has one $1$-handle, two $2$-handles, and one $3$-handle.

	\begin{figure}[ht]  \begin{center}
		\includegraphics[width=.6\textwidth]{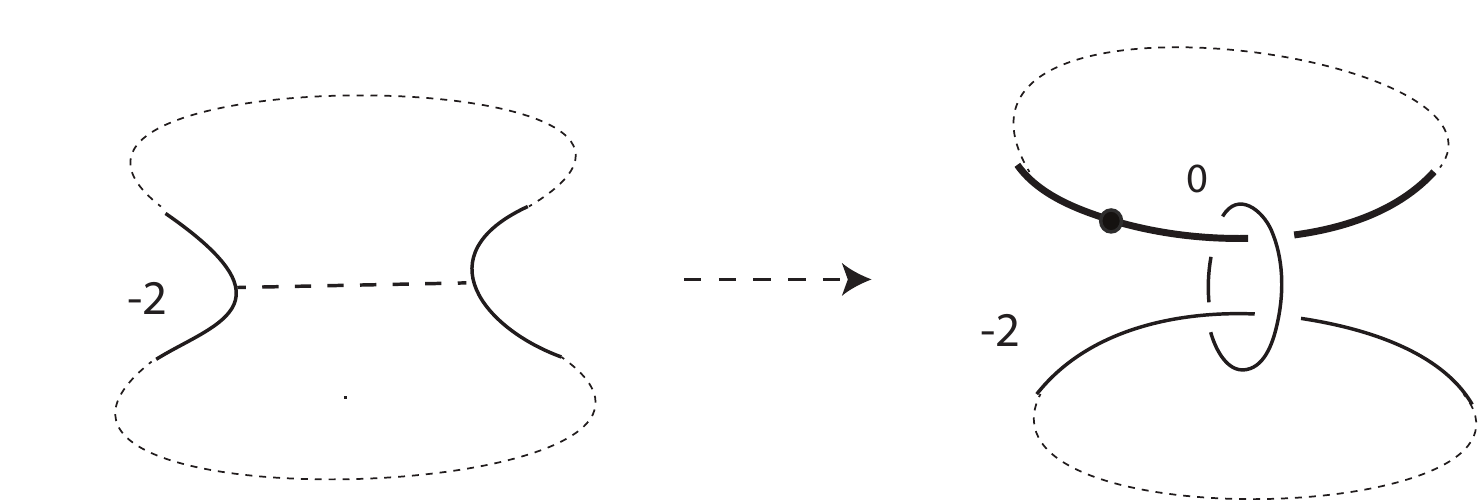}       
		\caption{}      \label{r1a} 
	\end{center}
\end{figure}	

	\begin{figure}[ht]  \begin{center}
			\includegraphics[width=1\textwidth]{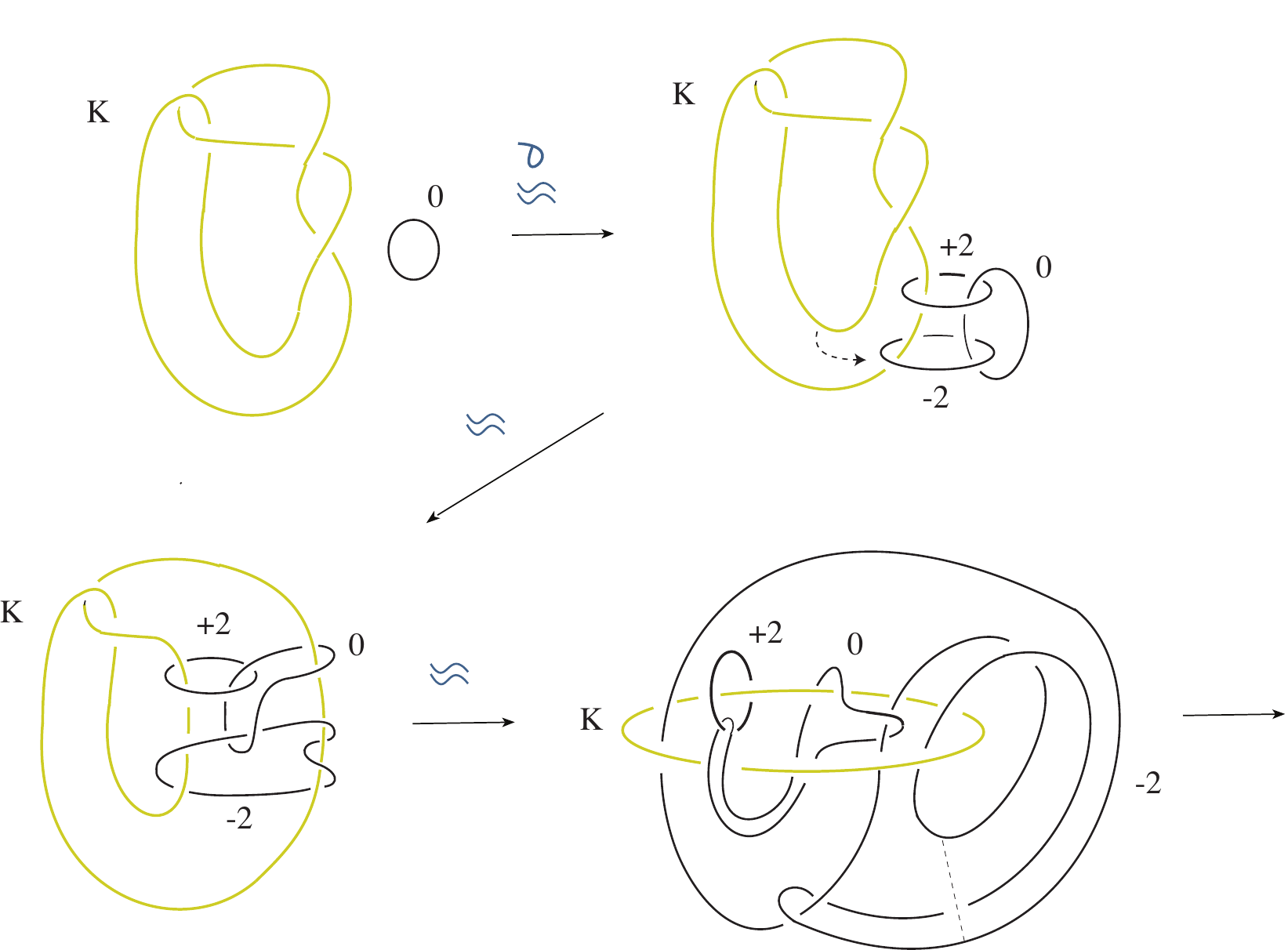}       
			\caption{}      \label{r1} 
		\end{center}
	\end{figure}
	
	\begin{figure}[ht]  \begin{center}
		\includegraphics[width=1\textwidth]{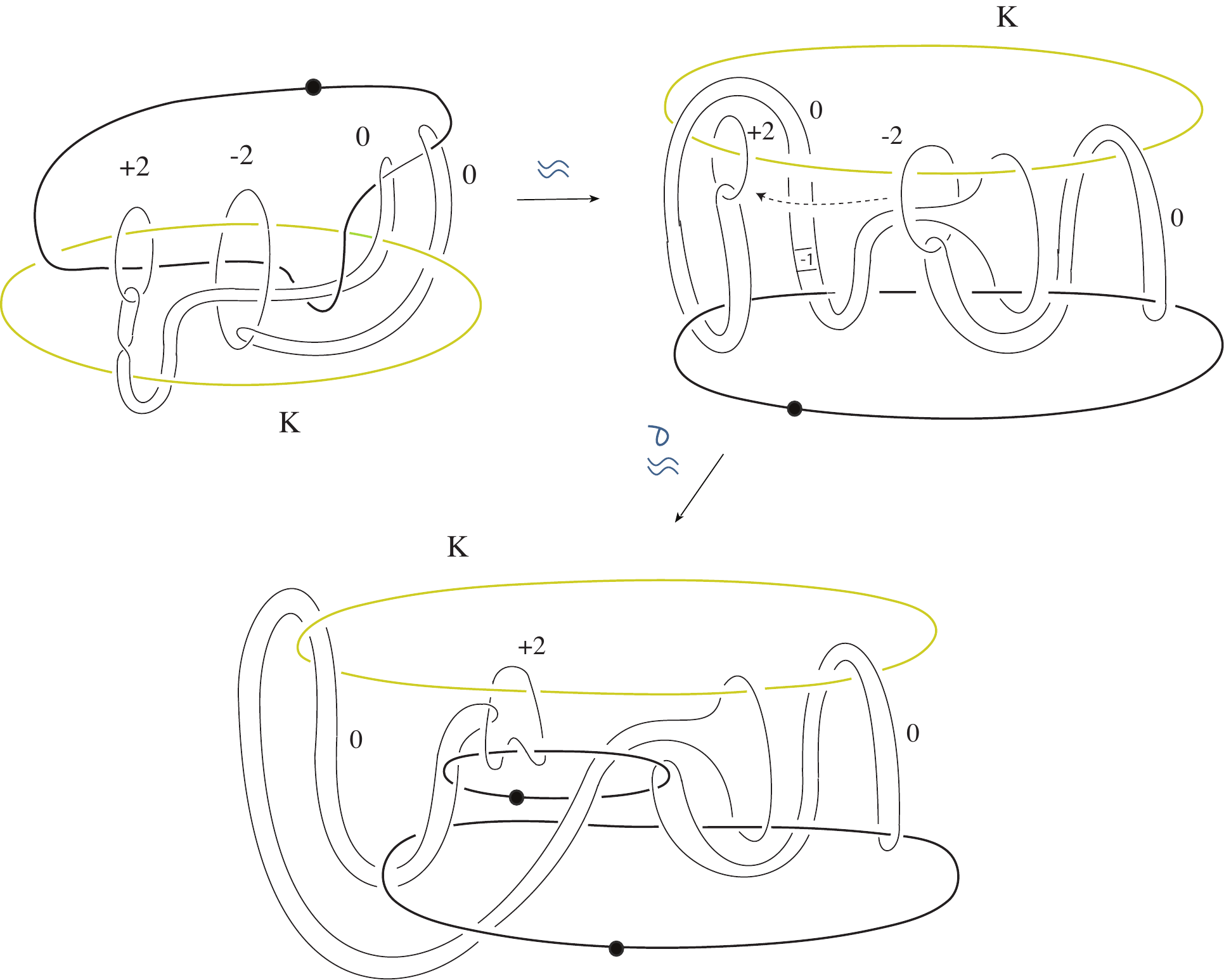}       
		\caption{$W_{0}$}      \label{r2} 
	\end{center}
\end{figure}

\begin{figure}[ht]  \begin{center}
		\includegraphics[width=.61\textwidth]{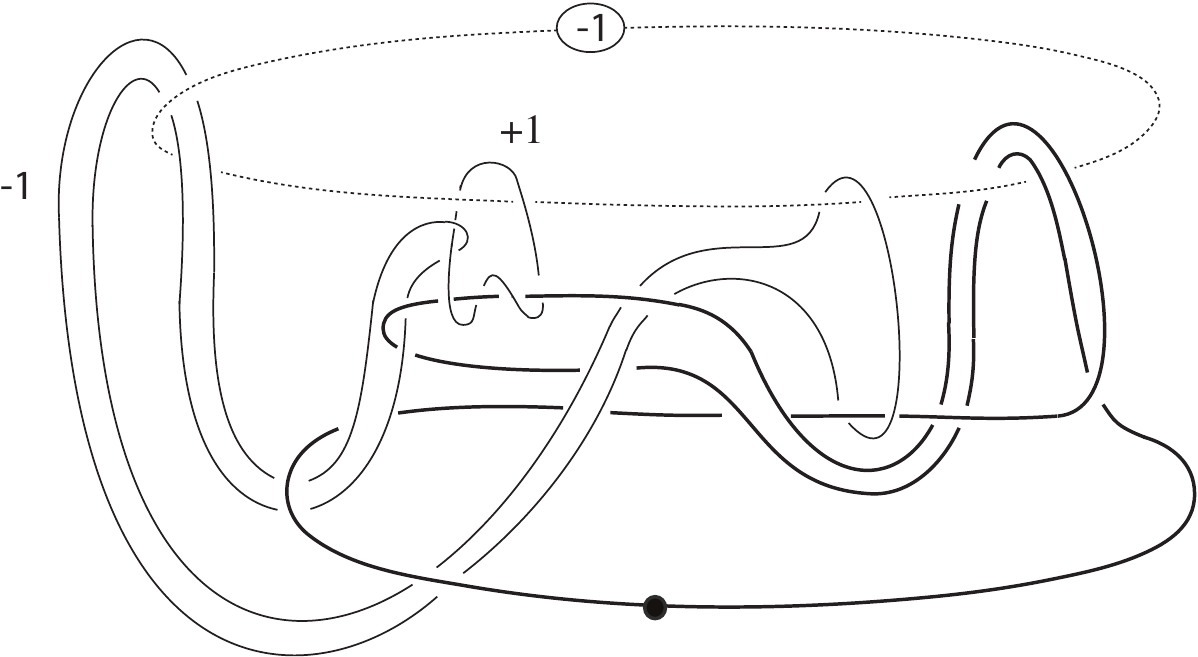}       
		\caption{W}      \label{r3} 
	\end{center}
\end{figure}

\end{section}


\bibliographystyle{amsalpha}
\bibliography{rational_balls.bib}

\end{document}